\documentclass[11pt, a4paper]{article}
\usepackage{amsmath,amsthm,amssymb, amsfonts, amscd}
\usepackage{calrsfs}

\newtheorem{theorem}{Theorem}[section]
\newtheorem{proposition}[theorem]{Proposition}
\newtheorem{corollary}[theorem]{Corollary}
\newtheorem{lemma}[theorem]{Lemma}

\theoremstyle{definition}
\newtheorem{remark}[theorem]{Remark}
\newtheorem{example}[theorem]{Example}
\newcommand{\dis}{\displaystyle}

\DeclareMathOperator{\spanned}{span}
\DeclareMathOperator{\card}{card}

\title{Multilinear H\"older-type inequalities on Lorentz sequence spaces}
\author{Daniel Carando \thanks{Dep. Matem\'atica, Fac. C. Exactas y Naturales, Universidad de Buenos Aires,%
Ciudad Universitaria, 1428, Buenos Aires, Argentina \texttt{dcarando@dm.uba.ar}
Partially supported by CONICET PIP 5272 and ANPCyT PICT 05 17-33042. Partially supported by UBACyT Grant X108 and ANPCyT PICT 06 00587.}%
\and Ver\'onica Dimant \thanks{Departamento de Matem\'{a}tica, Universidad de San Andr\'{e}s,%
Vito Dumas 284 (B1644BID) Victoria,Buenos Aires, Argentina.\texttt{vero@udesa.edu.ar}
Partially supported by CONICET PIP 5272 and ANPCyT PICT 05 17-33042}%
\and Pablo Sevilla-Peris  \thanks{Departamento de Matem\'atica Aplicada (ETSMRE) and IUMPA, Universidad Polit\'ecnica de Valencia,
Av. Blasco Ib\'a\~nez 21 46010, Valencia, Spain \texttt{psevilla@mat.upv.es}
Supported by the MECD Project MTM2005-08210 and grants GV-AEST06/092 and UPV-PAID-00-06.}}

\date{}

\begin{document}

\maketitle

\begin{abstract}
We establish H\"older type inequalities for Lorentz sequence spaces and their duals. In order to achieve these and some related inequalities,
we study diagonal multilinear forms in general sequence spaces, and obtain estimates for their norms. We also consider norms of multilinear
forms in different Banach multilinear ideals.
\end{abstract}

\section{Introduction}

Given a sequence $\alpha \in \ell_{\infty}$, the generalized H\"older's inequality affirms that,
for $1 \leq p \leq n$, there exists a constant $C>0$ such that for every $x_{1}, \ldots , x_{n} \in \ell_{p}$
\begin{equation} \label{Hoelder orig}
\bigg\vert \sum_{k=1}^{\infty} \alpha(k) x_{1}(k) \cdots  x_{n}(k) \bigg\vert
\leq C \Vert x_{1} \Vert_{\ell_{p}} \cdots  \Vert x_{n} \Vert_{\ell_{p}}.
\end{equation}
On the other hand, if $n < p < \infty$, again H\"older's inequality gives
that \eqref{Hoelder orig} holds if and only if $\alpha \in
\ell_{p/(p-n)}$. Moreover, it can be shown that the best constant
$C$ in \eqref{Hoelder orig} is in each case $\Vert \alpha
\Vert_{\ell_{\infty}}$ and $\Vert \alpha \Vert_{\ell_{p/(p-n)}}$. A natural
question now is if inequalities analogous to \eqref{Hoelder orig}
can be found in other Banach sequence spaces (see below for
definitions). More precisely, given $E$ a Banach sequence
space, under what conditions on $\alpha \in \ell_{\infty}$ there
exists $C>0$ such that for every $x_{1}, \ldots , x_{n} \in E$ the
following holds
\begin{equation} \label{Hoelder Bss}
\bigg\vert \sum_{k=1}^{\infty} \alpha(k) x_{1}(k) \cdots  x_{n}(k)
\bigg\vert \leq C \Vert x_{1} \Vert_{E} \cdots  \Vert x_{n}
\Vert_{E}?
\end{equation}
Our aim in this paper is to analyze the situation when $E$ is a Lorentz space $d(w,p)$ or a dual of a Lorentz space $d(w,p)^{*}$.
Then our two main results are

\begin{theorem} \label{lorentz}
Let $\alpha \in \ell_{\infty}$ and $E= d(w,p)$, then
\begin{enumerate}
\item[(a)] If $n \leq p$, then \eqref{Hoelder Bss} holds if and
only if $\alpha \in d(w,p/n)^{*}$.
\item[(b)] If $n>p$, then
\eqref{Hoelder Bss} holds if and only if
$\alpha\in m_\Psi$, where $m_\Psi$ is the Marcinkiewicz space
associated with  $\Psi(N)= \left( \sum_{k=1}^{N} w(k)
\right)^{n/p}$. If in addition $w$ is $n/(n-p)$-regular, then we
can change $m_\Psi$ by $\ell_\infty$.
\end{enumerate}
The best constant is $\|\alpha\|_{d(w,p/n)^{*}}$ in case {\rm (a)} and $\|\alpha\|_{m_\Psi}$ in case {\rm (b)}.
\end{theorem}

\begin{theorem} \label{duales 2}
Let $\alpha \in \ell_{\infty}$ and $E= d(w,p)^{*}$, then
\begin{enumerate}
\item[(a)] If $n' \leq p$, then \eqref{Hoelder Bss} holds if and
only if  $\alpha \in \ell_{\infty}$.
\item[(b)] If $n'>p>1$, then
\eqref{Hoelder Bss} holds if and only if  $\alpha \in
d(w^{\frac{n'}{n'-p}}, \frac{p'}{p'-n})$.
\item[(c)] If $p=1$, then
\eqref{Hoelder Bss} holds if and only if  $\alpha \in d(w^{n},1)$.
\end{enumerate}
The best constant in each case is the norm of $\alpha$ in the
corresponding space.
\end{theorem}

Our approach to this question is to study multilinear forms on the corresponding sequence spaces. Inequality~(\ref{Hoelder Bss}) can be
read as the continuity of the diagonal multilinear form on $E$ with coefficients $(\alpha(k))_k$. This way to look at H\"older inequalities
is crucial to our proofs of Theorems~\ref{lorentz} and \ref{duales 2}. Moreover, it motivates us to pose an analogous question in a more general
framework: if $\mathfrak{A}$ is a Banach ideal of multilinear mappings and $E$ is a Banach sequence space, under what conditions on
$\alpha \in \ell_{\infty}$ does the diagonal multilinear form with coefficients $(\alpha(k))_k$ belong to $\mathfrak A(^nE)$? As a
direct application of our results in this general framework, we consider nuclear and integral multilinear forms on Lorentz and dual of
Lorentz spaces.

\medskip
The article is organized as follows. In Section 2 we introduce notation, definitions and some general results. Sections~3 and~4 are
devoted to the proofs of Theorems~\ref{lorentz} and~\ref{duales 2}. In Section~5 we broaden the object of our study, considering diagonal
 multilinear forms belonging to different ideals defined on general sequence spaces. Combining this with the results of the previous
 sections we characterize the diagonal integral (and nuclear) multilinear forms on Lorentz sequence spaces and their duals.

\section{Preliminaries}
All through the paper we will use standard notation of the Banach space theory. We will consider complex Banach spaces $E,F, \ldots$ and
its duals will be denoted by $E^{*}, F^{*},\ldots$. Sequences of complex numbers will be denoted by $x = (x(k))_{k=1}^{\infty}$,
where each $x(k) \in \mathbb{C}$. By a Banach sequence space we will mean a Banach space
$E \subseteq \mathbb{C}^{\mathbb{N}}$ of sequences in $\mathbb{C}$ such that
$\ell_{1} \subseteq E \subseteq \ell_{\infty}$ satisfying that if $x \in \mathbb{C}^{\mathbb{N}}$ and $y \in E$ are such
that $\vert x(k) \vert \leq \vert y(k) \vert$ for all $k \in \mathbb{N}$ then $x \in E$ and $\Vert x \Vert \leq \Vert y \Vert$.
For each element in a Banach sequence space $x \in E$ its decreasing rearrangement $(x^{\star}(k))_{k=1}^{\infty}$ is given by
\[
x^{\star} (k) := \inf \{ \sup_{j \in \mathbb{N} \setminus J } | x(j)
|  \colon  J \subseteq \mathbb{N} \  , \ \card (J) < k \}.
\]
A Banach sequence space $E$ is called symmetric if $\| (x(k))_{k} \|_{E} = \|(x^{\star}(k))_{k} \|_{E}$ for every $x \in E$.
For each $N \in \mathbb{N}$ we consider the
$N$-dimensional truncation $E_{N}:= \spanned \{e_{1}, \dots,
e_{N}\}$ and we denote by $E_{0}$ the space of sequences in $E$ that
are all $0$ except for a finite number of coordinates. The canonical
inclusion $i_{N} : E_{N} \hookrightarrow E$ and projection $\pi_{N}
: E \rightarrow E_{N}$ are defined by $i_{N} ((x(k))_{k=1}^{N}) =
(x(1), \dots , x(N), 0,0 , \dots)$ and $\pi_{N}
((x(k))_{k=1}^{\infty}) = (x(k))_{k=1}^{N}$.\\
Given two Banach spaces, we will write $E=F$ if they are topologically isomorphic and $E\overset{1}{=}F$ if they are isometrically
isomorphic.\\

\medskip

The K\"othe dual of a Banach sequence space $E$ is defined as
\[
E^{\times} : = \{ z \in \mathbb{C}^{\mathbb{N}}  \colon \sum_{j \in
\mathbb{N}} |z(j) x(j)| < \infty \mbox{ for all } x \in E \}.
\]
This can be considered even if $E$ is not normed. If $E$ is quasi-normed, $E^{\times}$ with the norm $$\|z \|_{E^{\times}} :=
\sup_{\| x \|_{E} \leq 1} \sum_{j \in \mathbb{N}} |z(j) x(j)|$$ is a
Banach sequence space. It is easily seen that $z \in E^{\times}$ if and only if $\sum_{j \in \mathbb{N}} z(j) x(j)$ is finite
for all $x \in E$ and that $$\|z \|_{E^{\times}} = \sup_{\| x \|_{E} \leq 1} \Big|\sum_{j \in \mathbb{N}} z(j) x(j)\Big|.$$ Also,
$E^{\times}$ is symmetric whenever $E$ is symmetric. Note that
$(E_{N})^{*} \overset{1}{=} (E^{\times})_{N}$ holds for every $N$.\\

Following \cite[1.d]{LibroLiTz2}, a Banach sequence space
$E$ is said to be $r$-convex (with $1 \leq r < \infty$) if there
exists a constant $\kappa > 0$ such that for any choice $x_{1},
\dots, x_{m} \in E$ we have
\[
\bigg\| \bigg( \Big( \sum_{j=1}^{m} | x_{j}(k)|^{r} \Big)^{1/r} \bigg)_{k=1}^{\infty} \bigg\|_{E}
\leq \kappa \ \bigg( \sum_{j=1}^{m} \| x_{j}\|_{E}^{r} \bigg)^{1/r}
\]
On the other hand, $E$ is $s$-concave (with $1 \leq s < \infty$) if
there is a constant $\kappa > 0$ such that
\[
\bigg( \sum_{j=1}^{m} \| x_{j}\|_{E}^{s} \bigg)^{1/s}
\leq \kappa \ \bigg\| \bigg( \Big( \sum_{j=1}^{m} | x_{j}(k)|^{s} \Big)^{1/s}
\bigg)_{k=1}^{\infty} \bigg\|_{E}
\]
for all $x_{1}, \dots, x_{m} \in E$. We denote by $\mathbf{M}^{(r)} (E)$ and $\mathbf{M}_{(s)} (E)$ the smallest constants in each
inequality.\ Recall that $E$ is $r$-convex ($s$-concave) if and only
if $E^{\times}$ is $r'$-concave ($s'$-convex), where $r'$ and $s'$
are the conjugates of $r$ and $s$ respectively (see \cite[1.d.4]{LibroLiTz2}). Moreover, we have $M^{(r)}(E)=M_{(r')}(E^\times)$
($M_{(s)}(E)=M^{(s')}(E^\times)$ ).
If $E$ is $r$-convex for some $r$ or
$s$-concave for some $s$, then we say that $E$ has non-trivial convexity or non-trivial concavity.

Following standard notation, given a symmetric Banach sequence space
$E$ we consider the fundamental function of $E$:
\[
\lambda_{E} (N) := \Big\| \sum_{k=1}^{N} e_{k} \Big\|_{E}
\]
for $N \in \mathbb{N}$. For a detailed study and general facts of Banach sequence space, see \cite{LibroLiTz1, LibroLiTz2, TJ89}.\\

\begin{remark}\label{eleinf}
With this notation we can give a first positive answer to our question. If $E$ is $n$-concave, then $\alpha$ satisfies
\eqref{Hoelder Bss} if and only if $\alpha \in \ell_{\infty}$. Indeed, it is easily seen that being $n$-concave implies
$E \hookrightarrow \ell_{n}$ (given $x \in E$, just take $x_{k} = x(k) e_{k} \in E$ and apply the definition of concavity).
This and \eqref{Hoelder orig} immediately give that \eqref{Hoelder Bss} holds for any $\alpha \in \ell_{\infty}$.
\end{remark}

The space of continuous linear operators between two Banach
spaces $E,F$ will be denoted $\mathcal{L} (E;F)$ and the space of
continuous $n$-linear mappings $E_{1} \times \cdots \times E_{n}
\rightarrow F$ by $\mathcal{L} (E_{1}, \ldots, E_{n};F)$; with the
norm
\[
\Vert T \Vert := \sup \{ \Vert T(x_{1} ,  \ldots , x_{n}) \Vert_{F} \colon \Vert x_{i} \Vert_{E_{i}} \leq 1 \, , \, i=1, \ldots n\}
\]
this is a Banach space. If $E_{1}=\cdots = E_{n}=E$ we will write $\mathcal{L} (^{n} E;F)$ and
whenever $F=\mathbb{C}$ we will simply write $\mathcal{L}(E_{1}, \dots , E_{n})$ or $\mathcal{L} (^{n} E)$.

A mapping $P: E \rightarrow F$ is a continuous $n$-homogeneous
polynomial if there exists an $n$-linear mapping $T\in
\mathcal{L}(^{n}E;F)$ such that $P(x)=T(x, \ldots , x)$ for every
$x \in E$. The space of all continuous $n$-homogeneous polynomials
from $E$ to $F$ is denoted by $\mathcal{P}(^{n} E;F)$; endowed
with the norm $\Vert P \Vert = \sup_{\Vert x \Vert \leq 1 } \Vert
P(x) \Vert$ this is a Banach space. If $P$ is an $n$-homogeneous
polynomial and $T$ is the associated symmetric $n$-linear mapping,
then the polarization formula gives (see \cite[Proposition
1.8]{LibroDi99})
\begin{equation} \label{polariz}
\Vert P \Vert \leq \Vert T \Vert \leq \frac{n^{n}}{n!} \Vert P \Vert.
\end{equation}
A general study of the theory of polynomials on Banach spaces can be found in \cite{LibroDi99}.

\medskip
Ideals of multilinear forms were introduced in \cite{Pi84}.\ Let us
recall the definition.\ An ideal of multilinear forms $\mathfrak{A}$
is a  subclass of $\mathcal{L}$, the class of all multilinear forms
such that, for any  Banach spaces $E_{1}, \dots , E_{n}$ the set
\[
\mathfrak{A}(E_{1}, \dots , E_{n})=\mathfrak{A} \cap
\mathcal{L}(E_{1}, \dots , E_{n})
\]
satisfies
\begin{enumerate}
\item For any $\gamma_{1} \in E^{*}_{1}, \dots , \gamma_{n} \in
E^{*}_{n}$, the mapping $$(x_{1},\dots,x_{n}) \mapsto \gamma_{1}(x_{1})
\cdots \gamma_{n}(x_{n})$$ belongs to $\mathfrak{A}(E_{1}, \dots ,
E_{n})$.
\item If $S,T \in \mathfrak{A}(E_{1}, \dots , E_{n})$, then $S+T \in
\mathfrak{A}(E_{1}, \dots , E_{n})$.
\item If $T \in \mathfrak{A}(E_{1}, \dots , E_{n})$ and $S_{i} \in
\mathcal{L}(F_{i},E_{i})$ for $i=1, \dots ,n$, then $T \circ
(S_{1},\dots S_{n}) \in \mathfrak{A}(F_{1}, \dots , F_{n})$.
\end{enumerate}
An ideal of multilinear forms is called normed if for each
$E_{1},\dots , E_{n}$ there is a norm $\| \cdot \|_{\mathfrak{A}(E_{1}, \dots , E_{n})}$ in $\mathfrak{A}(E_{1},
\dots , E_{n})$ such that
\begin{enumerate}
\item $\|(x_{1},\dots,x_{n}) \mapsto \gamma_{1}(x_{1}) \cdots
\gamma_{n}(x_{n})\|_{\mathfrak{A}(E_{1}, \dots , E_{n})} =
\|\gamma_{1}\| \cdots \|\gamma_{n}\|$.
\item $\| T \circ (S_{1},\dots S_{n}) \|_{\mathfrak{A}(F_{1}, \dots , F_{n})} \leq
\|T\|_{\mathfrak{A}(E_{1}, \dots , E_{n})} \cdot \|S_{1}\| \cdots
\|S_{n}\|$.
\end{enumerate}
Analogously ideals of homogeneous polynomials were defined and studied in \cite{Fl01, Fl02, FlGa03, FlHu02}. However
\cite{FlGa03} shows that a polynomial is in a normed ideal of polynomials if and only if its associated multilinear mapping
is in some ideal of multilinear forms. Hence, dealing with one or the other type of ideals will not lead to essentially
different conclusions.

\vspace{.5cm}

If $(a(k))_{k}$ and $(b(k))_{k}$ are real sequences, we denote
$a(k) \prec b (k)$ when there exists $C>0$ such that $a(k) \leq C
b(k)$ for all $k \in \mathbb{N}$.\ Also, we denote $a(k) \asymp
b(k)$ when $a(k) \prec b(k)$ and $b(k) \prec a(k)$.

\section{Lorentz spaces} \label{Sect Lorentz}
Our aim in this section is to give the proof of Theorem \ref{lorentz}.\ Let us recall first the definition of Lorentz spaces; further
details and properties can be found in \cite[Section
4.e]{LibroLiTz1} and \cite[Section 2.a]{LibroLiTz2}.\ Let
$(w(k))_{k=1}^{\infty}$ be a decreasing sequence of positive
numbers such that $w(1)=1$, $\lim_{k} w(k)=0$ and
$\sum_{k=1}^{\infty} w(k)=\infty$.\ Then the corresponding Lorentz
sequence space, denoted by $d(w,p)$ is defined as the space of
all sequences $(x(k))_{k}$ such that
\[
\| x \| = \sup_{\pi \in \Sigma_{\mathbb{N}}} \bigg(
\sum_{k=1}^{\infty} |x({\pi(k))}|^{p} w(k) \bigg)^{1/p} =
\bigg( \sum_{k=1}^{\infty} |x^{\star}(k)|^{p} w(k) \bigg)^{1/p}
<\infty
\]
where $\Sigma_{\mathbb{N}}$ denotes the group of permutations of the natural numbers.
Each $d(w,p)$ is clearly a symmetric Banach sequence space.\\
The sequence $w$ is said to be $\alpha$-regular ($0< \alpha < \infty$) if $w(k)^{\alpha} \asymp
\frac{1}{k} \sum_{j=1}^{k} w(j)^{\alpha}$ and regular if it is $\alpha$-regular
for some $\alpha$.\\
In \cite{Re81} it can be found that $d(w,p)$ is  $r$-convex (and
$\mbox{\bf M}^{(r)} (d(w,p)) =1$) whenever  $1 \leq r \leq p$. Also
\cite[Theorem 2]{Re81} shows that, for $p<s<\infty$, $d(w,p)$ is $s$-concave if
and only if $w$ is $\frac{s}{s-p}$-regular. It is non-trivially
concave if and only if $w$ is 1-regular.

In \cite{Ga69} and \cite{LibroLiTz1} a description of $d(w,p)^{*}$, the dual of $d(w,p)$,
is given as the space of those sequences $x$ such that there exists  a decreasing $y
\in B_{\ell_{p'}}$ with
\[
\sup_{N} \frac{\sum_{k=1}^{N} x^{\star}(k)}{\sum_{k=1}^{N}
y(k) w(k)^{1/p}} < \infty
\]
for $p>1$. The norm in $d(w,p)^{*}$ is the infimum of the expression
above over all possible decreasing $y \in B_{\ell_{p'}}$. For $p=1$,
\[
d(w,1)^{*}=\Big\{ x \colon \| x \| = \sup_{N}
\frac{\sum_{k=1}^{N} x^{\star}(k)}{\sum_{k=1}^{N} w(k)} <
\infty \Big\}.
\]

If $w$ is regular, an easier description of $d(w,p)^{*}$ with $p>1$ can be given. In this case we
have in \cite{All78} and \cite{Re82} that
\[
d(w,p)^{*} =
\Big\{ x \colon \bigg( \frac{x^{\star}(k)}{w(k)^{1/p}} \bigg)_{k=1}^{\infty} \in \ell_{p'} \Big\}.
\]
The $\ell_{p'}$ norm of this sequence is a positive homogeneous
function of $x$ which, although not a norm, is equivalent to
the norm in $d(w,p)^{*}$ (see \cite[Theorem 1]{Re82}).\\

\medskip
Lorentz spaces $d(w,p)$ are reflexive whenever $p>1$ \cite[Sect 4.e]{LibroLiTz1}. If
$p=1$ the predual of $d(w,1)$ can be described as (see \cite{Sa60, Ga66})
\[
d_{*}(w,1)=\left\{ x \in c_{0} \colon \lim_{N \to \infty}
\frac{\sum_{k=1}^{N} x^{\star}(k)}{\sum_{k=1}^{N} w(k)} = 0 \right\}
\]
with the norm $\|x\| = \sup_{N} \frac{\sum_{k=1}^{N} x^{\star}(k)}
{\sum_{k=1}^{N} w(k)}$.

Let us recall that, given a strictly positive, increasing sequence
$\Psi$ such that $\Psi(0)=0$, the associated Mar\-cinkiewicz
sequence space $m_\Psi$ (see \cite[Definition 4.1]{KaLe06}, also
\cite{ChHa06,KaLe04}) consists of all sequences $(x(k))_{k}$ such
that
\[
\Vert x \Vert_{m_{\Psi}} = \sup_{N} \frac{\sum_{k=1}^{N}
x^{\star}(k)} {\Psi (N)} < \infty .
\]

In order to prove part (a) of Theorem \ref{lorentz} we make use of
a general result. Let us recall first that if $E$ is a symmetric
Banach sequence space, its $n$-concavification $E_{(n)}$ (see
\cite[Section 1.d]{LibroLiTz2}) is defined as the set consisting of
those sequences $(z(k))_{k}$ so that $(\vert z(k) \vert^{1/n}
)_{k} \in E$. On $E_{(n)}$ we can define a symmetric quasi-norm by
$\Vert z \Vert_{E_{(n)}}= \Vert (\vert z(k)
\vert^{1/n} )_{k} \Vert_{E}^n$. This quasi-norm verifies the ``monotonicity condition": if $z \in \mathbb{C}^{\mathbb{N}}$ and $w \in E_{(n)}$
are such
that $\vert z(k) \vert \leq \vert w(k) \vert$ for all $k \in \mathbb{N}$ then $z \in E_{(n)}$ and $\Vert z \Vert_{E_{(n)}} \leq
\Vert w \Vert_{E_{(n)}}$.
If $E$ is $n$-convex and $\mathbf{M}^{(n)} (E)=1$, then $\Vert \cdot \Vert_{E_{(n)}} $ is actually a norm and $E_{(n)}$ turns out to be
a symmetric Banach sequence space.

We can now give the result we need. This could be deduced from a result on orthogonally additive polynomials on Banach lattices given in
\cite[Theorem 2.3]{BeLaLla06}. However, in our setting (symmetric Banach sequence spaces) it is easier to give a direct proof. Note that the
K\"othe dual is by definition
the set in which we have some H\"older inequality. In
\eqref{Hoelder Bss} we aim to an $n$-linear H\"older inequality; it is no surprise then that the K\"othe dual of
$E_{(n)}$ appears.

\begin{lemma} \label{diag L n-cvx}
Let $\alpha \in \ell_{\infty}$ and $E$ be a symmetric Banach sequence space, then \eqref{Hoelder Bss} holds
if and only if $ \alpha \in (E_{(n)})^{\times}$ and
the best constant in \eqref{Hoelder Bss} is $\Vert \alpha \Vert_{(E_{(n)})^{\times}}$.
\end{lemma}
\begin{proof}
Let $\alpha$ satisfy \eqref{Hoelder Bss}; then there exists $C>0$ such that for every $x \in E$,
\[
\Big\vert \sum_{k=1}^{\infty} \alpha(k) x(k)^{n} \Big\vert \leq C \Vert x \Vert^{n}.
\]
This implies that $\sum_{k} \alpha(k) z(k)$ is finite for every $z
\in E_{(n)}$ hence $\alpha \in (E_{(n)})^{\times}$ and $\Vert \alpha \Vert_{(E_{(n)})^{\times}} \leq C$.

On the other
hand, if $\alpha \in (E_{(n)})^{\times}$ let us take $x_{1} ,
\ldots , x_{n} \in E$. Note first that the inequality
\begin{equation}\label{medias}
(|x_{1}(k)| \cdots |x_{n}(k)|)^{1/n} \leq \frac{|x_{1}(k)| +
\cdots + |x_{n}(k)|}{n}
\end{equation}
implies that $\big((x_{1}(k) \cdots x_{n}(k))^{1/n}\big)_k \in E$
and then $z:=\big(x_{1}(k) \cdots x_{n}(k)\big)_k \in E_{(n)}$. As a consequence of \eqref{medias} we have
$\|z\|_{E_{(n)}} \leq \|x_{1}\|_{E} \cdots \|x_{n}\|_{E}$. Therefore
\begin{multline*}
\Big\vert \sum_{k=1}^{N}  \alpha(k) x_{1} (k) \cdots x_{n}(k) \Big\vert
\leq \sum_{k=1}^{N} | \alpha(k) x_{1} (k) \cdots x_{n}(k)| \\
\leq \| \alpha \|_{(E_{(n)})^{\times}} \|z\|_{E_{(n)}}
\leq \| \alpha \|_{(E_{(n)})^{\times}} \|x_{1}\|_{E} \cdots \|x_{n}\|_{E}
\end{multline*}
holds for every $N$. Thus \eqref{Hoelder Bss} is verified with $C=\| \alpha \|_{(E_{(n)})^{\times}}$ and this completes the proof.
\end{proof}

The last inequality in the previous proof can be seen as an estimation of the norm of a multilinear form. Let us say that a multilinear form
$T$ on a sequence space $E$ is called diagonal if there exists a sequence $\alpha$ such
that for every $x_{1}, \ldots, x_{n} \in E$
\[
T(x_{1}, \ldots, x_{n}) = \sum_{k=1}^{\infty} \alpha (k) x_{1}(k)
\cdots x_{n}(k).
\]
In this case we write $T=T_{\alpha}$. With this terminology, Lemma~\ref{diag L n-cvx} states that diagonal $n$-linear forms on $E$ correspond to
sequences $\alpha \in (E_{(n)})^\times$ and $$\|T_\alpha\|=\|\alpha\|_{(E_{(n)})^\times}.$$

The
$n$-homogeneous polynomial associated to $T_{\alpha}$ is also called
diagonal and is denoted $P_\alpha$.

\begin{remark} \label{igualnorm}
We observe in (\ref{polariz}) the general relationship between the
norms of a polynomial and its associated symmetric $n$-linear
form. For diagonal forms and polynomials defined on a  symmetric
Banach sequence space $E$ the situation is different.  It is
proved in the previous lemma that, if $x_1,\dots, x_n$ are in $E$
then $\big(\big(x_1(k)\cdots x_n(k)\big)^{1/n}\big)_k$ also
belongs to $E$ and
\[
\|\big(\big(x_1(k)\cdots x_n(k)\big)^{1/n}\big)_k\|^n\leq
\|x_1\|\dots\|x_n\|
\]
Then, the norm of any multilinear diagonal form on $E$ coincides
with the norm of its associated diagonal polynomial, that is
$\|T_\alpha\|=\|P_\alpha\|$.
\end{remark}

Lemma~\ref{diag L n-cvx} provides an abstract characterization of the sequences $\alpha$ such that inequality~(\ref{Hoelder Bss}) is verified.
However, the K\"othe dual of the $n$-concavification of $E$ is not always the simplest way to obtain an explicit description of such sequences.
Therefore, in some cases we will use different approaches.

Now we prove our first theorem.

\subsection*{Proof of Theorem \ref{lorentz}}
For the statement (a), since $n\le p$, the $n$-concavification of $d(w,p)$ is the space $d(w,p/n)$. Then Lemma \ref{diag L n-cvx} gives the conclusion.

For the statement (b), let $\alpha$ and $C>0$ satisfy
\eqref{Hoelder Bss} with $E=d(w,p)$. For any fixed $N \in
\mathbb{N}$, let $J_{N} \subseteq \mathbb{N}$ be such that
$|J_{N}|=N$ then for any $\lambda_{1}, \dots , \lambda_{n} \in
\mathbb{C}$ with $|\lambda_{k}|=1$,
\[
\Big\vert \sum_{k\in J_{N}} \alpha(k) \lambda_{k}^{n} \Big\vert
\leq C \Big\Vert \sum_{k\in J_{N}} \lambda_{k} e_{k}
\Big\Vert_{d(w,p)}^{n} =  C \Big(\sum_{k=1}^{N} w(k) \Big)^{n/p}.
\]
Choosing $\lambda_{k}$ and $J_{N}$ so that $\sum_{k\in J_{N}}
\lambda_{k}^{n} \alpha(k) = \sum_{k=1}^{N} \alpha^{\star}(k)$ we
get, for any $N$,
\[
\frac{\sum_{k=1}^{N} \alpha^{\star}(k)} {\big( \sum_{k=1}^{N} w(k) \big)^{n/p}} \leq C.
\]
Thus, $\alpha\in m_\Psi$, with $\Psi(N)=\big( \sum_{k=1}^{N} w(k)
\big)^{n/p}$.

For the reverse inclusion, let $\alpha\in m_\Psi$. Without loss
of generality we can assume $\alpha = \alpha^{\star}$. Let us
consider the diagonal $n$-linear mapping $T_{\alpha} : d(w,p)
\times \cdots \times d(w,p) \rightarrow \mathbb{C}$. By Remark
\ref{igualnorm}, $T_{\alpha}$ is continuous if and only if the
associated polynomial $P_{\alpha}:d(w,p)\to\ell_n$ is continuous,
and their norms are equal. First of all
\[
|P_{\alpha} (x)| = \Big| \sum_{k=1}^{\infty} \alpha(k) x(k)^n \Big| \leq \sum_{k=1}^{\infty} \alpha(k) \ x^{\star}(k)^{n}.
\]
If we prove that
\begin{equation} \label{induc1}
\sum_{k=1}^{N} \alpha(k) \ x^{\star}(k)^{n} \leq
\|\alpha\|_{m_\Psi} \Big(\sum_{k=1}^{N} w(k) \  x^{\star}(k)^{p}
\Big)^{n/p}
\end{equation}
holds for every $N$, then we will have $|P_{\alpha}(x)| \leq
\|\alpha\|_{m_\Psi} \|x\|_{d(w,p)}^{n}$
and the result will follow.

We can assume that $x=x^{\star}$. By the definition of $m_\Psi$ we
have
\begin{align*}
\sum_{k=1}^{N} \alpha(k) x(k)^{n} & = \sum_{i=1}^{N-1} \Big(
\sum_{k=1}^{i} \alpha(k) \Big) (x(i)^{n}-x(i+1)^{n})
+ \Big( \sum_{k=1}^{N} \alpha(k) \Big) x(N)^{n} \\
& \le \|\alpha\|_{m_\Psi} \sum_{i=1}^{N} \Psi(i)
(x(i)^{n}-x(i+1)^{n})
+ \|\alpha\|_{m_\Psi} \Psi(N) x(N)^{n} \\
& = \|\alpha\|_{m_\Psi}\Big[ \Psi(1) x(1)^{n} + \sum_{i=2}^{N}
\big(\Psi(i)-\Psi(i-1)\big) x(i)^{n}\Big].
\end{align*}
To obtain \eqref{induc1}, we need to prove that for every $N$, the following inequality holds:
\begin{equation} \label{induc2}
\Psi(1) x(1)^{n} + \sum_{i=2}^{N} \big(\Psi(i)-\Psi(i-1)\big)
x(i)^{n} \leq \Big( \sum_{k=1}^{N} w(k) \  x(k)^{p} \Big)^{n/p}.
\end{equation}
We proceed by induction. For $N=1$, the inequality is obvious. By
the induction hypothesis we have
\begin{equation*}\begin{split}
\Psi(1) x(1)^{n} & + \sum_{i=2}^{N+1} \big(\Psi(i)-\Psi(i-1)\big)
x(i)^{n} \\ & \leq \Big( \sum_{k=1}^{N} w(k) \  x(k)^{p} \Big)^{n/p} +
\big(\Psi(N+1)-\Psi(N)\big) x(N+1)^{n}.
\end{split}\end{equation*}
We want to show that the last expression is at most
$\big( \sum_{k=1}^{N+1} w(k) \  x(k)^{p} \big)^{n/p}$.
Equivalently, we have to prove
\begin{equation}\begin{split}
&\Psi(N+1)-\Psi(N) \\ & \leq \left( \sum_{k=1}^{N+1} w(k) \
\Big(\frac{x(k)}{x(N+1)}\Big)^{p} \right)^{n/p} -\left(
\sum_{k=1}^{N} w(k) \ \Big(\frac{x(k)}{x(N+1)}\Big)^{p}
\right)^{n/p}.\label{menor}
\end{split}\end{equation}
Consider the increasing function $\phi(t)=(t+w(N+1))^{n/p} -
t^{n/p}$ (recall that $n\ge p$). Since $x$ is decreasing,
$\sum_{k=1}^{N} w(k) \leq \sum_{k=1}^{N} w(k)
\Big(\frac{x(k)}{x(N+1)}\Big)^p$. Hence
\[
\phi \Big(\sum_{k=1}^{N} w(k)\Big) \leq \phi \Big(\sum_{k=1}^{N}
w(k) \Big(\frac{x(k)}{x(N+1)}\Big)^p \Big),
\]
but this is exactly what we want in \eqref{menor}.  This gives
\eqref{induc2}, hence \eqref{induc1} holds and the result
follows.

If in addition $w$ is $n/(n-p)$-regular, then it is easy to see
that $m_\Psi$ is isomorphic to $\ell_\infty$. This completes the proof.
\qed

\begin{remark} \label{obs2}
It is known (and can be deduced, for example, from
\cite[Lemma~3.3]{KaLe04}) that $m_\Psi$ is isomorphic to the dual
of a Lorentz space $d(\overline w,1)$ for some sequence $\overline
w$, understanding  $d(\overline w,1)=\ell_1$ if $\overline w$ is
not a null sequence.

In some cases, the sequence $\overline w$ can be determined. For
example, for $\breve{w}(k)=\Psi(k)-\Psi(k-1)$, we have
$$\frac{\sum_{k=1}^{N} \alpha^{\star}(k)}{\Psi(N)}
= \frac{\sum_{k=1}^{N} \alpha^{\star}(k)}{\sum_{k=1}^{N} \breve{w}(k)}.$$
If
$\breve{w}$ is decreasing, we obtain that \eqref{Hoelder Bss}
holds for $E=d(w,p)$ if and only if $\alpha \in
d(\breve{w},1)^{*}$. Moreover, there are universal constants
$A_{n}, B_{n}$  (not depending on $\alpha$) so that the best $C>0$
in \eqref{Hoelder Bss} satisfies $A_{n} \Vert \alpha
\Vert_{d(\breve{w},1)^{*}} \leq C \leq B_{n} \Vert \alpha
\Vert_{d(\breve{w},1)^{*}}$.

If $w$ is regular (i.e., 1-regular) and the sequence
$\tilde{w}(k)=\frac{(k w(k))^{n/p}}{k}$ is decreasing we get
another description, namely $m_\Psi = d(\tilde{w},1)^{*}$. Indeed,
by the mean value theorem
\[
\Psi(k)-\Psi(k-1) = \frac{n}{p} z(k)^{n/p-1} w(k)
\]
for some $ \sum_{j=1}^{k-1} w(j) \leq z(k) \leq
 \sum_{j=1}^{k} w(j) $. But $ \sum_{j=1}^{k} w(j)
\asymp k w(k)$ and $\sum_{j=1}^{k-1} w(j) \asymp (k-1) w(k-1) \geq
(k-1) w(k) \succ k w(k)$. So we have $z(k) \asymp (k{w}(k))$.
Consequently, $\breve{w}(k) \asymp (k w(k))^{n/p-1} w(k) =
\tilde{w}(k)$ and, since $(\tilde{w}(k))_{k}$ is decreasing, we have that
$m_\Psi = d(\tilde{w},1)^{*}$. Hence, in this case, \eqref{Hoelder
Bss} holds if and only if $\alpha \in d(\tilde{w},1)^{*}$.

Note that $\breve{w}(k) \asymp \tilde{w}(k)$ if and only if $w$ is
regular. Also, if either  $\breve w$ or $\tilde w$ is decreasing but does not
converge to zero, then the corresponding Lorentz space $d(\cdot,
1)$ is in fact $\ell_1$ and then its dual is $\ell_\infty$.
\end{remark}

In the following example we apply our results to the Lorentz
sequence spaces $\ell_{p,q}$. For the particular case $q<n<p$,
this example shows that the regularity condition in part (b) of
Theorem~\ref{lorentz} is sharp: for any $r<n/(n-p)$ there are
$r$-regular sequences $w$ so that \eqref{Hoelder Bss} does not
hold for some $\alpha \in \ell_{\infty}$ and $E=d(w,p)$.

\begin{example}\label{lpq}
Special cases of Lorentz sequence spaces are the $\ell_{p,q}$
spaces. For $p>q\ge 1$ these spaces are defined as
\[
\ell_{p,q}= \Big\{x: \|x\|= \Big(\sum_{k=1}^\infty \frac{(x^\star(k))^q}{k^{1-\frac q p}}\Big)^{1/q} <\infty \Big\}.
\]
The space $\ell_{p,q}$ is the Lorentz sequence space $d(w,q)$ with $w(k)=k^{q/p-1}$.

We apply the above results to these particular spaces. By part (a)
of Theorem~\ref{lorentz}, we obtain for $n\le q$, that \eqref{Hoelder Bss} holds for $E=\ell_{p,q}$ if and only if
$\alpha \in (\ell_{\frac p n,\frac q n})^*$.

If $n\ge p$, since $\ell_{p,q}\hookrightarrow\ell_n$, we have that
\eqref{Hoelder Bss} holds if and only if $\alpha \in \ell_\infty$.

Finally, for $q<n<p$ we can apply part (b) of
Theorem~\ref{lorentz}. However, since $w$ is regular and $\tilde
w(k)=\frac{(kw(k))^{n/q}}k=k^{n/p-1}$ is a decreasing sequence,
Remark~\ref{obs2} gives that  \eqref{Hoelder Bss} holds if and
only if $\alpha \in d(\tilde w,1)^*=(\ell_{\frac p n ,1})^*$.

It is easy to check that the sequence $(k^{q/p-1})_k$ is
$r$-regular if and only if $r<p/(p-q)$. Therefore, for any
$r<n/(n-q)$ we can find $p>n$ such that $r<p/(p-q)$. In this case,
the sequence associated to $\ell_{p,q}$ is $r$-regular but
\eqref{Hoelder Bss} does not hold for some $\alpha \in
\ell_\infty$.
\end{example}

\section{Duals of Lorentz spaces} \label{Sect duales}
We give now the proof of Theorem \ref{duales 2}. We have seen in
Section~\ref{Sect Lorentz} that using $n$-linear diagonal forms
can sometimes be helpful.  In the same spirit, an operator $D \in
\mathcal{L} (E;F)$ between Banach sequence spaces is called
diagonal if there exists a sequence $\sigma$ such that
$D(x)=(\sigma(k) x(k))_{k=1}^{\infty}$; in this case we write
$D=D_{\sigma}$. Some relationship between diagonal operators and
diagonal $n$-lineal forms is shown in the following lemma, that we
will need later.

\begin{lemma} \label{norma lambda}
Let $E$ be a symmetric Banach sequence space and $T_{\alpha} : E
\times \cdots \times E \rightarrow \mathbb{C}$ a diagonal
$n$-linear form. Let $D_{\sigma} : E \rightarrow \ell_{n}$ be the
diagonal operator associated to $\sigma = \alpha^{1/n}$
(coordinatewise). Then $T_{\alpha}$ is continuous if and only if
$D_{\sigma}$ is continuous and
\[
\| T_{\alpha} \|=\| D_{\sigma} \|^{n} .
\]
\end{lemma}
\begin{proof}
If $P_\alpha$ is the $n$-homogeneous polynomial associated to
$T_\alpha$, by Remark~\ref{igualnorm}, we have that $\|T_{\alpha}
\|=\|P_\alpha\|\leq \| D_{\sigma} \|^{n}$.

On the other hand, if $|\lambda(k)|=1$ for all $j$, then
$\|(\lambda(k) x(k))_{k}\|_{E} = \|x\|_{E}$ and
\[
\| T_{\alpha} \|  \geq \sup_{\|x\|_{E}\leq 1 \atop \alpha(k)x(k)^{n}\geq 0} \Big| \sum_{k=1}^{\infty} \alpha(k)\  x(k)^{n} \Big|
=  \sup_{\|x\|_{E}\leq 1} \sum_{k=1}^{\infty} |\alpha(k)| \ |x(k)|^{n} = \| D_{\sigma} \|^{n}.
\]
\end{proof}

\noindent Now we are ready to proof our theorem for duals of Lorentz spaces.
\subsection*{Proof of Theorem \ref{duales 2}}
Part (a) follows from Remark \ref{eleinf} and the fact that $d(w,p)^{*}$ is
$n$-concave if and only if $d(w,p)$ is $n'$-convex and this
happens if and only if $1 \leq n' \leq p$. In this case we
have that the $n$-concavity constant $\mathbf M_{(n)}(d(w,p)^*)$ is $1$. Since the norm of a diagonal multilinear form coincides with the
norm of its associated polynomial, the best constant is
$\|\alpha\|_\infty$.

To get part (b), let us take $\alpha \in \ell_{\infty}$ and $\sigma =
\alpha^{1/n}$. If $D_{\sigma} : d(w,p)^{*} \rightarrow \ell_{n}$
is the diagonal operator associated to $\sigma$ and $D'_{\sigma} :
\ell_{n'} \rightarrow d(w,p)$ is the adjoint operator, we want to
show that
\begin{equation} \label{D-alpha}
\| D'_{\sigma} \|= \| \alpha \|_{d(w^{\frac{n'}{n'-p}}, \frac{p'}{p'-n})}^{1/n}.
\end{equation}
If this is the case, then by Lemma \ref{norma lambda}
\[
\| \alpha \|_{d(w^{\frac{n'}{n'-p}}, \frac{p'}{p'-n})} = \| D'_{\sigma} \|^{n}
= \| D_{\sigma} \|^{n} = \Vert T_{\alpha} \Vert
\]
and for every $x_{1}, \ldots , x_{n} \in d(w,p)^{*}$,
\[
\Big\vert \sum_{k} \alpha (k) x_{1}(k) \cdots x_{n}(k) \Big\vert
\leq \| \alpha \|_{d(w^{\frac{n'}{n'-p}}, \frac{p'}{p'-n})} \Vert x_{1} \Vert \cdots \Vert x_{n} \Vert.
\]
Hence, \eqref{Hoelder Bss} holds if and only if $\alpha \in d(w^{\frac{n'}{n'-p}}, \frac{p'}{p'-n})$ and the best constant is the norm of
$\alpha$ in this space. Let us then  show that \eqref{D-alpha} holds. First,
\[
\| D'_{\sigma} (x) \| = \| ( \sigma (k) x(k))_{k} \|_{d(w,p)}
=\sup_{\pi \in \Sigma_{\mathbb{N}}} \Big( \sum_{k} \big|
\alpha(\pi(k))^{1/n} x(\pi(k)) \big|^{p} w(k) \Big)^{1/p}.
\]
Using H\"older's inequality with $n'/p$ and $n'/(n'-p)$ we obtain,
for each $\pi \in \Sigma_{\mathbb{N}}$,
\begin{equation*}\begin{split}
\Big( \sum_{k}   \big| \alpha(\pi(k))^{1/n}  &  x(\pi(k)) \big|^{p} w(k)
\Big)^{1/p} \\ & \leq \Big( \sum_{k} |x(\pi(k))|^{n'} \Big)^{1/n'}
\Big( \sum_{k} | \alpha(\pi(k)) |^{\frac{p'}{p'-n}} w(k)^{\frac{n'}{n'-p}} \Big)^{\frac{n'-p}{n'p}} \\
& \leq \| x \|_{\ell_{n'}} \ \Big( \sum_{k} \alpha^{\star}(k)^{\frac{p'}{p'-n}} w(k)^{\frac{n'}{n'-p}} \Big)^{\frac{p'-n}{p'n}}.
\end{split}\end{equation*}
Hence $\| D'_{\sigma} \| \leq \| \alpha \|_{d(w^{\frac{n'}{n'-p}}, \frac{p'}{p'-n})}^{1/n}$. Let us see now
that this value is attained. Since all the involved spaces are symmetric we can assume without loss of generality that
$\alpha = \alpha^{\star}$. Then let us consider
\[
x_{N}(k) = \frac{\alpha(k)^{\frac{p}{(n'-p)n}}
w(k)^{\frac{1}{n'-p}}} {\left( \sum_{i=1}^{N}
\alpha(i)^{\frac{p'}{p'-n}} w(i)^{\frac{n'}{n'-p}} \right)^{1/n'}}
\]
for $k=1, \dots , N$. It is easily seen that $\|
(x_{N}(k))_{k=1}^{N} \|_{\ell_{n'}}=1$ and
\begin{align*}
\| D'_{\sigma} (x_{N}) \|_{d(w,p)}
= \Big( \sum_{k=1}^{N} \alpha(k)^{\frac{p'}{p'-n}} & w(k)^{\frac{n'}{n'-p}} \Big)^{1/p-1/n'} \\
& = \Big\| \sum_{k=1}^{N} \alpha(k) e_{j} \Big\|_{d(w^{\frac{n'}{n'-p}}, \frac{p'}{p'-n})}^{1/n}.
\end{align*}
Therefore $\big\| \sum_{k=1}^{N} \alpha(k) e_{j} \big\|_{d(w^{\frac{n'}{n'-p}}, \frac{p'}{p'-n})}^{1/n} \leq \|D'_{\sigma}\|$
for all $N$ and the result follows.

Statement (c) follows similarly. \qed

\section{A general approach}
We have seen in Sections \ref{Sect Lorentz} and \ref{Sect duales} that considering diagonal $n$-linear forms helps
in proving H\"older-type inequalities. In fact, if in \eqref{Hoelder Bss} we take the supremum over $\Vert x_{i} \Vert_{E} \leq 1$,
$i=1, \ldots , n$ then we have that the best constant in \eqref{Hoelder Bss} is precisely $\Vert T_{\alpha} \Vert$. We see that
our problem is closely related with determining the norm of diagonal $n$-linear forms. This sits very much in the
philosophy of considering norms of diagonal multilinear forms in different ideals presented in \cite{CaDiSe06, CaDiSe07} and
motivates us to broaden our framework.

\medskip
Following \cite{Ko75} for the linear case and \cite{CaDiSe07} for
the multilinear case, if $\mathfrak{A}$ is a Banach ideal of
multilinear mappings we consider, for each $n \in \mathbb{N}$, the
space
\[
\ell_{n}(\mathfrak{A},E):=\{\alpha\in \ell_\infty : T_{\alpha} \in \mathfrak{A}(^n  E)\}.
\]
With the norm $\|\alpha\|_{\ell_{n}(\mathfrak{A},E)}=\|T_{\alpha}\|_{\mathfrak{A}(^n E)}$ this
is a symmetric Banach sequence space whenever $E$ is so.\\
If $\mathcal{L}$ denotes the ideal of all multilinear forms, then \eqref{Hoelder orig} can be rewritten as
\[
\ell_{n}(\mathcal{L},\ell_{p}) \stackrel{1}=
\begin{cases}
\ell_{\infty} & \text{ if } 1\leq p \leq n \\
\ell_{p/(p-n)} & \text{ if } n<p<\infty
\end{cases}
\]
and our results Theorem \ref{lorentz} and \ref{duales 2} can be summarized as
\begin{gather*}
\ell_{n} (\mathcal{L}, d(w,p))\stackrel{1}=
\begin{cases}
 d(w,p/n)^{*} & \text{ if } n \leq p \\
 m_{\Psi} & \text{ if } n>p
\end{cases}
\\
\ell_{n} (\mathcal{L}, d(w,p)^{*}) \stackrel{1}=
\begin{cases}
 \ell_{\infty} & \text{ if } n' \leq p \\
 d(w^{\frac{n'}{n'-p}}, \frac{p'}{p'-n}) & \text{ if } n'>p>1\\
 d(w^n,1) & \text{ if } p=1
\end{cases}
\end{gather*}
where $\Psi(N)=\big( \sum_{j=1}^{N} w(j) \big)^{n/p}$. If $n>p$
and $w$ is $n/(n-p)$-regular, then $\ell_{n} (\mathcal{L},
d(w,p))=\ell_{\infty}$.\\

Our aim in this section is to obtain descriptions of $\ell_{n} ( \mathfrak{A} , d(w,p))$ and $\ell_{n} ( \mathfrak{A} , d(w,p)^{*})$
for ideals other than $\mathcal{L}$. We will make use of some general facts. If $E$ is a Banach sequence space, we consider
the mapping $\Phi_{N} : E_{N}\times \cdots \times E_{N} \longrightarrow \mathbb{C}$ given by
\[
\Phi_{N}(x_{1},\dots , x_{n})= \sum_{k=1}^{N} x_{1}(k) \cdots x_{n}(k).
\]
Clearly $\|\Phi_N\|_{\mathfrak{A}(^nE)}=\lambda_{\ell_{n}(\mathfrak{A},E)}(N)$.

If $F$ and $G$ are symmetric Banach sequence spaces so that $F
\hookrightarrow  G$ then we have, by the closed graph theorem,
\[
\lambda_{G}(N) \prec \lambda_{F}(N).
\]
A weak converse of this fact can be obtained under certain assumptions. We need first a lemma.

\begin{lemma} \label{lema tecnico}
Let $F$ and $G$ be symmetric Banach sequence spaces and suppose
there exists $\alpha
>0$ be such that $\lambda_G(N) \prec
\lambda_{F}(N)^{\alpha}$. Then, for all $\varepsilon
>0$ we have $\left(
\dis\frac{1}{k^{\varepsilon} \lambda_{F}(k)^{\alpha}} \right)_{k \in
\mathbb{N}}\in G$ .
\end{lemma}
\begin{proof}
For each $m \in \mathbb{N}\cup \{0\}$, we define $\mathbb{N}_{m} =
\{ k \in \mathbb{N} \ : \  2^{m} \leq k < 2^{m+1} \}$ and
\[
x_m = \sum_{k \in \mathbb{N}_{m}} e(k) .
\]
Since $G$ is symmetric, $\|x_{m}\|_G = \lambda_G(2^m) \prec
\lambda_{F}(2^{m})^{\alpha}$. Hence,
\[
\sum_{m} \frac{1}{2^{m \varepsilon} \lambda_{F}(2^{m})^{\alpha}}\
x_{m} \in G.
\]
Now, for $k \in \mathbb{N}_{m}$, we have $1/k \leq 1/2^{m}$ and
$1/\lambda_{F}(k) \leq 1/\lambda_{F}(2^{m})$ and the result follows.
\end{proof}

\begin{proposition} \label{inclus vuelta}
Let $F$ and $G$ be a symmetric Banach sequence spaces for which
there exists $0<\varepsilon<1$ such that $\lambda_G(N)\prec
\lambda_F(N)^{1-\varepsilon}$. If $F$ satisfies $N^{\delta}\prec
\lambda_F(N)$ for some $\delta> 0$, then we have $F\hookrightarrow G$.
\end{proposition}
\begin{proof}
Let $x \in F$.\ We can assume, without loss of generality, that
$x(k)=x^{\star}(k)$ is decreasing.\ Then
\[
x(k) \lambda_{F}(k) \leq \Big\| \sum_{j=1}^{k} x(j) e_{j}\Big\|_F
\leq \|x \|_{F}.
\]
Now,
$\lambda_{F}(k)=\lambda_{F}(k)^{\varepsilon}\lambda_{F}(k)^{1-\varepsilon}
\succ  \ k^{\varepsilon\delta} \lambda_{F}(k)^{1-\varepsilon}$.\
Hence
\[
x(k) \prec \frac{\|x\|_{F}}{k^{\varepsilon\delta}
\lambda_{F}(k)^{1-\varepsilon}}.
\]
By Lemma \ref{lema tecnico}, $x \in G$.
\end{proof}

Note that the additional condition on the sequence space $F$ is
automatically satisfied whenever $F$ or $G$ have non-trivial
concavity. The previous results can be reformulated to obtain
information on the space $\ell_n(\mathfrak{A},E)$.

\begin{corollary}
Let $E, F$ and $G$ be symmetric Banach sequence spaces and
$\mathfrak{A}$ be a Banach ideal of multilinear forms.
\begin{enumerate}
\item[(a)] If $F\hookrightarrow \ell_n(\mathfrak{A},E) \hookrightarrow G$,
then $\lambda_{G}(N) \prec \| \Phi_{N} \|_{\mathfrak{A}(^{n}E)}
\prec
\lambda_{F}(N)$.
\item[(b)] If there exists $\varepsilon >0$ such that
$\|\Phi_{N}\|_{\mathfrak{A}(^{n}E_{N})} \prec
\lambda_{F}(N)^{1-\varepsilon}$ and $F$ has non-trivial concavity,
then $F \hookrightarrow
\ell_{n}(\mathfrak{A}, E)$.
\item[(c)] If there exists $\varepsilon >0$ such that
$\lambda_{G}(N)^{1+\varepsilon} \prec
\|\Phi_{N}\|_{\mathfrak{A}(^{n}E_{N})}$ and $G$ has non-trivial
concavity, then $\ell_{n}(\mathfrak{A}, E) \hookrightarrow G$.
\end{enumerate}
\end{corollary}

\bigskip

If $\mathfrak{A}$ is a normed ideal of $n$-linear forms, the
maximal hull $\mathfrak{A}^{\mathrm{max}}$ of $\mathfrak{A}$ is defined as
the class of all $n$-linear forms $T$ such that
\begin{align*}
\|T\|_{\mathfrak{A}^{\mathrm{max}}(E_{1},\dots,E_{n})}
:=\sup\{  \|T \mid_{M_{1} \times \cdots \times M_{n}} & \|_{\mathfrak{A}(M_{1}, \dots ,M_{n})}: \\
& M_{i} \subset E_{i} ,\  \dim M_{i}<\infty \}
\end{align*}
is finite. $\mathfrak{A}^{\mathrm{max}}$ is always complete and it is the largest ideal whose norm coincides with
$\|\cdot\|_{\mathfrak{A}}$ in finite dimensional spaces. A normed ideal  $\mathfrak{A}$ is called maximal if
$(\mathfrak{A},\|\cdot\|_{\mathfrak{A}})=(\mathfrak{A}^{\mathrm{max}},\|\cdot\|_{\mathfrak{A}^{\mathrm{max}}})$. Maximal ideals are those
whose norm is uniquely determined by finite dimensional subspaces.

It is a well known fact that the space of $n$-linear forms on a finite dimensional space $M$ can be identified with the
$n$-fold tensor product $\bigotimes^{n} M^{*}$ by identifying each tensor $\gamma_{1} \otimes \cdots \otimes \gamma_{n}$ with the mapping
$(x_{1} , \ldots , x_{n}) \rightsquigarrow \gamma_{1} (x_{1}) \cdots  \gamma_{n}(x_{n})$. Then the ideal norm induces a tensor
norm $\eta$ on $\bigotimes^{n} M^{*}$ (the tensor product with this norm is denoted by  $\bigotimes_{\eta}^{n} M^{*}$). By a standard
procedure the norm $\eta$ can be extended from tensor norms in the
class of finite dimensional normed spaces to the class of all normed spaces. In this case, the tensor norm $\eta$ and the ideal
$\mathfrak{A}$ are said to be associated. A detailed study of the subject and presentation
of the procedure can be found in \cite{LibroDeFl93, Fl01, Fl02, FlGa03, FlHu02}.

Given a normed ideal $\mathfrak{A}$ associated to the finitely
generated tensor norm $\alpha$, its adjoint ideal $\mathfrak{A}^*$
is defined by
\[
\mathfrak{A}^*(^{n} E):=\big(\textstyle\bigotimes_{\eta}^{n} E \big)^{*}.
\]
The adjoint ideal is called dual ideal in \cite{Fl01}. The tensor norm associated to $\mathfrak{A}^*$ is denoted by
$\eta^*$. We also have the representation theorem \cite[Section 3.2]{FlHu02} (see also \cite[Section 4]{Fl01}):
\[
\mathfrak{A}^{\mathrm{max}}(^{n} E)= \big(\textstyle\bigotimes_{\eta^*}^{n} E \big)^{*}.
\]
\noindent In particular, this shows that the adjoint ideal $\mathfrak{A}^*$ is maximal.\\

For a maximal ideal $\mathfrak{A}$, the space
$\ell_{n}(\mathfrak{A}, E)$ coincides isometrically with
$\ell_{n}(\mathfrak{A}, E^{\times \times})$. This fact is a
consequence of the following lemma, which we believe is of independent interest.

\begin{lemma} \label{density lemma}
Let $E$ be a symmetric Banach sequence space and $\mathfrak{A}$ a
maximal Banach ideal of multilinear forms.\ Let $T:E \times \cdots
\times E \rightarrow \mathbb{C}$ $n$-linear and suppose there exists $C>0$ such that, for every $N \in
\mathbb{N}$, the restriction $T^{N}$ to $E_{N} \times \cdots
\times E_{N}$ satisfies $\|T^{N}\|_{\mathfrak{A}(^{n} E_{N})}
\leq C$.\ Then $T \in \mathfrak{A}(^{n} E)$ and
$\|T\|_{\mathfrak{A}(^{n} E)} \leq C$.
\end{lemma}
\begin{proof}
Since $\mathfrak{A}$ is maximal, there exists a finitely generated
tensor norm $\nu$ such that $(\bigotimes_{\nu}^{n} E)^{*} =
\mathfrak{A}(^{n}E)$.\ Since $E$ is a symmetric space, both the
inclusion $i_{N}:E_{N} \hookrightarrow E_{0}$ and the projection
$\pi_{N} :E_{0} \rightarrow E_{N}$ have norm $1$.\ These, together
with the metric mapping property, give that the mapping
$\bigotimes_{\nu}^{n} E_{N} \hookrightarrow
\bigotimes_{\nu}^{n} E_{0}$ is an isometry onto its image.

Let now $s \in \bigotimes_{\nu}^{n} E_{0}$; then $s \in
\bigotimes_{\nu}^{n} E_{N}$ for some $N$ and
\[
|T(s)|=|T^{N}(s)| \leq C \ \nu(s,\textstyle\bigotimes^{n} E_{N}) = C \
\nu(s,\bigotimes^{n} E_{0}).
\]
Hence $T|_{\bigotimes_{\nu}^{n} E_{0}} \in
(\bigotimes_{\nu}^{n} E_{0})^{*}$.\ Since $\bigotimes_{\nu}^{n}
E_{0}$ is dense in $\bigotimes_{\nu}^{n} E$, by the Density
Lemma \cite[13.4]{LibroDeFl93}, $T \in (\bigotimes_{\nu}^{n}
E)^{*}=\mathfrak{A}(^{n}E)$ and  $\|T\|_{\mathfrak{A}(^{n} E)} \leq
C$.
\end{proof}
The previous lemma also holds if $E$ is a Banach space with an
unconditional basis with constant $K$.\ Indeed, in this case $\|
\pi_{N} \| \leq K$, and $\nu (s , \bigotimes^{n} E_{0}) \leq \nu
(s , \bigotimes^{n} E_{N})
\leq K^{n} \nu (s , \bigotimes^{n} E_{0})$.\ Then $\|T\|_{\mathfrak{A}(^{n} E)} \leq C \ K^{n}$.\\

\begin{proposition} \label{doble kothe}
Let $E$ be a symmetric Banach sequence space and $\mathfrak{A}$ a
maximal Banach ideal of multilinear forms. Then
$$\ell_n(\mathfrak{A},E)\overset1=\ell_n(\mathfrak{A},E^{\times\times})$$
\end{proposition}
\begin{proof}
Since $E$ is contained in $E^{\times\times}$ with a norm one
inclusion, it is immediate that
$\ell_n(\mathfrak{A},E^{\times\times})\subset
\ell_n(\mathfrak{A},E)$ (with norm one inclusion).

Conversely, let $\alpha \in \ell_n(\mathfrak{A},E)$. For each $N$,
$\|T^N_\alpha\|_{\mathfrak{A}(^{n} E_N)}\le
\|T_\alpha\|_{\mathfrak{A}(^{n} E)}$. Since
$E_N\overset1=(E_N)^{\times\times}\overset1=(E^{\times\times})_N$,
we have $\|T^N_\alpha\|_{\mathfrak{A}(^{n} (E^{\times\times})_N)}\le
\|T_\alpha\|_{\mathfrak{A}(^{n} E)}$. By Lemma~\ref{density lemma},
$T_\alpha$ belongs to $\mathfrak{A}(^{n} E^{\times\times})$ and
$\|T_\alpha\|_{\mathfrak{A}(^{n} E^{\times\times})} \le
\|T_\alpha\|_{\mathfrak{A}(^{n} E)}$.
\end{proof}
The ideal $\mathcal{L}$ of all multilinear forms is obviously maximal; then by Theorem \ref{duales 2} (c) we have the
following reformulation of \cite[Theorem 2.5]{LouPe06}
\[
\ell_{n}(\mathcal{L}, d_{*} (w,1)) \overset{1}{=} d(w^{n},1).
\]

Let us recall the trace duality between
$\mathfrak{A}^{*}(^{n}E^{\times}_{N})$ and
$\mathfrak{A}(^{n}E_{N})$. Suppose $T \in
\mathfrak{A}^{*}(^{n}E^{\times}_{N})$ can be written as a finite sum
of the form
\[
T(\gamma_{1}, \dots , \gamma_{n})= \sum_{j} \gamma_{1}(x_1^{j})
\cdots \gamma_{n}(x_n^{j})
\] and $S\in \mathfrak{A}(^{n}E_{N})$ is of the form
\[
S(x_{1}, \dots , x_{n})= \sum_{i} \gamma_1^{i}(x_{1}) \cdots
\gamma_n^{i}(x_{n}).
\]
Then, the duality is given by
\begin{equation} \label{ii}
\begin{split}
\langle T,S \rangle  =   \sum_{i,j} \gamma_1^{i}(x_1^{j}) & \cdots
\gamma_n^{i}(x_n^{j}) \\
& = \sum_{i} T(\gamma_1^{i}, \dots ,
\gamma_n^{i}) = \sum_{j} S(x_1^{j}, \dots , x_n^{j}).
\end{split}
\end{equation}

The following finite dimensional identifications are easy to check.
These will enable us to prove a duality result in the proposition
below.
\begin{gather}
\ell_{n}(\mathfrak{A}, E_{N}) \stackrel{1}{=} [\ell_{n}(\mathfrak{A}, E)]_{N}  \label{i} \\
 \mathfrak{A}(^{n}E_{N})^{*} \stackrel{1}{=}
 \mathfrak{A}^{*}(^{n}E^{\times}_{N}) \stackrel{1}{=}
\mathfrak{A}^{*}(^{n}E^{*}_{N})  \label{iii} \\
\ell_{n}(\mathfrak{A}, E)^{\times}_{N} \stackrel{1}{=} \ell_{n}(\mathfrak{A}, E_{N})^{\times}
\stackrel{1}{=} \ell_{n}(\mathfrak{A}^{*}, E^{\times}_{N})
\stackrel{1}{=} \ell_{n}(\mathfrak{A}^{*}, E^{\times})_{N}  \label{iv}
\end{gather}

\begin{proposition} \label{dual ele}
Let $E$ be a symmetric Banach sequence space and $\mathfrak{A}$ a Banach ideal of multilinear forms; then
\[
\ell_{n} (\mathfrak{A},E)^{\times} \stackrel{1}{=} \ell_{n}(\mathfrak{A}^{*},E^{\times}).
\]
\end{proposition}
\begin{proof}
Let us take first $\alpha \in \ell_{n} (\mathfrak{A},E)^{\times}$;
then the associated $n$-linear form  $T_{\alpha}$ is defined on
the space of finite sequences in $E^{\times}$.\ Moreover, using
\eqref{iv}, we have
\begin{multline*}
\|T_{\alpha}|_{E_{N}^{\times} \times \cdots \times E_{N}^{\times}}
\|_{\mathfrak{A}^{*}(^{n}E_{N}^{\times})}  = \| \pi_{N} (\alpha)
\|_{\ell_{n}(\mathfrak{A}^{*},E_N^{\times})} \\
= \| \pi_{N} (\alpha) \|_{\ell_{n}(\mathfrak{A}, E)^{\times}_{N}}
\leq \| \alpha \|_{\ell_{n}(\mathfrak{A}, E)^{\times}}.
\end{multline*}
By Lemma \ref{density lemma}, $\alpha$ belongs to
$\ell_{n}(\mathfrak{A}^{*},E^{\times})$ and $\| \alpha
\|_{\ell_{n}(\mathfrak{A}^{*}, E^{\times})} = \| T_{\alpha}
\|_{\mathfrak{A}^{*}(^{n} E^{\times})}
\leq \| \alpha \|_{\ell_{n}(\mathfrak{A}, E)^{\times}}$.\\

We take now $\alpha \in \ell_{n}(\mathfrak{A}^{*},E^{\times})$ and a
norm one $\beta \in \ell_{n}(\mathfrak{A},E)$.\ For each $j$, let
$\tilde\beta(j)$ be such that $\alpha(j) \tilde\beta(j)=|\alpha(j)
\beta(j)|$. Then, by symmetry and \eqref{ii}
\begin{multline*}
\sum_{j=1}^{N} |\alpha(j) \beta(j)|
= \sum_{j=1}^{N} \alpha(j) \tilde\beta(j) = \langle T_{\pi_{N}(\alpha)} ,
T_{\pi_{N}(\tilde\beta)} \rangle_{\mathfrak{A}^{*}(^{n}E_{N}^{\times}), \mathfrak{A}(^{n}E_{N})} \\
\leq \| T_{\alpha}\|_{\mathfrak{A}^{*}(^{n}E^{\times})} \
\|T_{\tilde\beta}\|_{\mathfrak{A}(^{n}E)}= \|
T_{\alpha}\|_{\mathfrak{A}^{*}(^{n}E^{\times})} \
\|T_{\beta}\|_{\mathfrak{A}(^{n}E)}
=\| \alpha \|_{\ell_{n}(\mathfrak{A}^{*},E^{\times})}.
\end{multline*}
And this completes the proof.
\end{proof}

By applying Proposition~\ref{dual ele} to the adjoint ideal and to
the K\"othe dual of $E$ and Proposition~\ref{doble kothe}  we get
\[
\ell_{n} (\mathfrak{A}^{*}, E^{\times})^{\times} = \ell_{n}
(\mathfrak{A}^{**}, E^{\times \times}) = \ell_{n}
(\mathfrak{A}^{\mathrm{max}}, E^{\times \times})= \ell_{n}
(\mathfrak{A}^{\mathrm{max}}, E)
\]
isometrically. Therefore, if $\mathfrak{A}$ is maximal we
immediately have
\[
\ell_{n} (\mathfrak{A},E)\overset 1 = \ell_{n} (\mathfrak{A}^{*}, E^{\times})^{\times}.
\]

In view of Proposition \ref{dual ele} we can use  Theorem \ref{lorentz}  and Theorem \ref{duales 2} to get results
on ideals other than $\mathcal{L}$. Let us recall that $T \in\mathcal{L} (^{n} E)$ is called nuclear if
there are sequences $(\gamma_{1,k})_{k}, \dots, (\gamma_{n,k})_{k}$ in
$E^{*}$ with $\| \gamma_{i,k}\| \leq 1$ for all $k$ and $i =1,\dots,n$
and there is $(\lambda(k))_{k} \in\ell_{1}$ so that for every
$x_{1},\dots, x_{n} \in E$
\[
T(x_{1}, \dots, x_{n})= \sum_{k} \lambda(k)\cdot \gamma_{1,k}(x_{1}) \cdots \gamma_{n,k}(x_{n}) .
\]

We denote by $\mathcal{N}$ the ideal of nuclear forms. The nuclear norm is defined as the infimum of
$\sum_{k} \Vert \lambda(k) \Vert  \Vert \gamma_{1,k} \Vert \cdots \Vert \gamma_{n,k} \Vert$ over all possible representations.\\

A mapping $T \in\mathcal{L} (^{n} E)$ is called integral if
there exists a positive Borel-Radon measure $\mu$ on $B_{E^{*}}
\times\cdots\times B_{E^{*}}$ (with the weak$^{*}$-topologies) such
that
\[
T(x_{1}, \dots, x_{n})= \int_{B_{E^{*}} \times\cdots\times B_{E^{*}}}
\gamma_{1}(x_{1}) \cdots \gamma_{n}(x_{n}) \ d \mu(\gamma_{1}, \dots, \gamma_{n})
\]
for all $x_{1},\dots, x_{n} \in X$ (see \cite[4.5]{LibroDeFl93}
and \cite{Al85}).\ The ideal of integral multilinear forms
is denoted by $\mathcal{I}$. It is well known that $\mathcal{L}^{*}=\mathcal{I}$. We then have

\begin{gather*}
\ell_{n} (\mathcal{I}, d(w,p)) \stackrel{1}= \left\{
\begin{array}{ll}  d(w^n,1)^* & \text{ if } p=1 \\
& \\
d(w^{\frac{n'}{n'-p}}, \frac{p'}{p'-n})^{*} & \text{ if  } 1<p<n'  \\
& \\
\ell_{1} & \text{ if  } n' \leq p
\end{array}
\right. \\
 \\
\ell_{n} (\mathcal{I}, d(w,p)^{*}) \stackrel{1}= \left\{
\begin{array}{ll}
m_\Psi^{\times}=(m^0_\Psi)^{*} & \text{ if  } 1\le p<n \\
& \\
d(w, p/n) & \text{ if  } n \leq p  \\
\end{array}
\right.
\end{gather*}
Here $m^0_\Psi$ denotes the subspace of order continuous elements of $m_\Psi$, and verifies  $(m^0_\Psi)^{**}=m_\Psi$
(see \cite{KaLe04}). The equality $m_\Psi^{\times}=(m^0_\Psi)^{*}$ follows from the proof of \cite[Theorem 3.4]{KaLe04}.

Whenever a space $E$ is reflexive or has a separable dual, nuclear
and integral mappings on $E$ coincide. Therefore, for $1<p<\infty$,
$\ell_{n} (\mathcal{I}, d(w,p))=\ell_{n} (\mathcal{N}, d(w,p))$ and
$\ell_{n} (\mathcal{I}, d(w,p)^*)=\ell_{n} (\mathcal{N}, d(w,p)^*)$.
Also, $\ell_{n} (\mathcal{N}, d_*(w,1))=\ell_{n} (\mathcal{I},
d_*(w,1))=\ell_{n} (\mathcal{I}, d^*(w,1)) $ (the last equality
follows from Proposition~\ref{doble kothe}).

By Remark \ref{obs2}, for $p<n$,
$\ell_{n} (\mathcal{I}, d(w,p)^{*})$  can be identified isomorphically with
$d(\overline{w},1)^{**}$ for some $\overline{w}$. Moreover, if $p<n$
and  $w$ is $n/(n-p)$-regular, then $\ell_{n} (\mathcal{I},
d(w,p)^{*})=\ell_1$ by Theorem \ref{lorentz}.

\begin{remark}
We have already mentioned that $\|\Phi_N\|_{\mathfrak{A}(^nE)}=\lambda_{\ell_{n}(\mathfrak{A},E)}(N)$ always holds.
Therefore, all the previous results immediately give estimations for the usual
and the nuclear norms of $\Phi_{N}$ (the nuclear and integral norms
of $\Phi_{N}$ always coincide).

Moreover, these estimates have an immediate tensor counterpart, since
$\|\Phi_{N}\|_{\mathcal{L}(^{n}E)} = \| \sum_{j=1}^{N} e'_{j}
\otimes \cdots \otimes e'_{j} \|_{\bigotimes_{\varepsilon}^{n} E'}$
and $\|\Phi_{N}\|_{\mathcal{N}(^{n}E)} = \| \sum_{j=1}^{N} e'_{j}
\otimes \cdots \otimes e'_{j} \|_{\bigotimes_{\pi}^{n} E'}$ ($\varepsilon$ and $\pi$ denote respectively the injective and
projective tensor norms).
\end{remark}

\section*{Acknowledgements}
We would like to thank Silvia Lassalle and Andreas Defant for helpful conversations and
suggestions that improved the final shape of the paper.

Most of the work in this article was performed while the third cited author was
visiting the Departments of Mathematics at Universidad de Buenos Aires and Universidad de San Andr\'es during
the summer/winter of 2006 supported by grants GV-AEST06/092 and UPV-PAID-00-06. He
wishes to thank all the people in and outside both Departments that made that such a
delightful time.

\end{document}